\documentclass[10pt]{amsart}
\usepackage{adjustbox}   
\usepackage{amsmath, amssymb,mathtools,amsthm,amscd,
multirow, tikz-cd,blindtext, graphicx, float, tabularx, lipsum}
\usepackage[normalem]{ulem}
\usepackage[all]{xy}
\usepackage[english]{babel}
\definecolor{dodgerblue}{rgb}{0.12, 0.56, 1.0}
\usepackage[raiselinks=false,colorlinks=true,citecolor=blue,urlcolor=dodger blue,
linkcolor=blue,bookmarksopen=true,pdftex]{hyperref}

\tikzcdset{scale cd/.style={every label/.append style={scale=#1},
    cells={nodes={scale=#1}}}}
\usepackage[mathscr]{eucal}
\newtheorem{theorem}{Theorem}[section]
\newtheorem{proposition}[theorem]{Proposition}
\newtheorem{lemma}[theorem]{Lemma}
\newtheorem{remark}[theorem]{Remark}
\newtheorem{definition}[theorem]{Definition}

\newtheorem{corollary}[theorem]{Corollary}

\newcommand{\lra}[1]{ \langle #1 \rangle}

\newcommand{\ga}{\gamma}

\newcommand{\bb}{\mathbb}

\newcommand{\Index}{\mbox{Index}}

\newcommand{\bea} {\begin{eqnarray*}}
\newcommand{\beq} {\begin{equation}}
\newcommand{\bey} {\begin{eqnarray}}
\newcommand{\eea} {\end{eqnarray*}}
\newcommand{\eeq} {\end{equation}}
\newcommand{\eey} {\end{eqnarray}}

\newcolumntype{P}[1]{>{\centering\arraybackslash}p{#1}}
\title{Some observations on projective Stiefel Manifolds}
\author[B. Kundu]{Bikramjit Kundu}
\address{Department of Mathematics, Indian Institute of Science Education and Research Bhopal, Bhopal Bypass Road, 462 066, Madhya Pradesh, India}
\email{bikramju@gmail.com, bikramjitkundu@iiserb.ac.in}
\author[S. Podder]{Sudeep Podder}
\address{Indian Statistical Institute, 8th Mile, Mysore Road, RVCE Post, Bangalore 560059, India}
\email{sudeeppodder1993@gmail.com\\ sudeep\_pd@isibang.ac.in}

\date{}
\keywords{Projective Stiefel manifolds, Flip Stiefel manifolds, Upper characteristic rank.}
\subjclass[2020]{Primary: 57R20}

\begin{document}
\begin{abstract}
    In this note, we compute the upper characteristic rank of the projective Stiefel manifolds over $\bb R, \bb C$ and $\bb H$ and the flip Stiefel manifolds. We also provide bounds for the $\mathrm{cup}$ lengths of these spaces. We also provide
necessary conditions for the existence of $S^3$-map between quaternionic Stiefel manifolds using the Fadell-Husseini index.
\end{abstract}
\maketitle

\section{Introduction}
It is well known that the mod two cohomology algebra of the real projective space $\bb RP^n$ is generated by the first Stiefel-Whitney class of the tautological line bundle over $\bb RP^n$. If $n$ is odd, the mod two cohomology algebra of the unoriented Grassmannian $G_{n,k}$ of $k$-planes in $\mathbb R^n$ is generated by the Stiefel-Whitney classes of its tangent bundle (see \cite[Lemma 2.1]{ps1988}). It is natural to ask whether the cohomology algebra of a finite CW complex $Y$ is generated by the Stiefel-Whitney classes of some vector bundle over $Y$, or, at least, what is the largest integer $r\leq \dim (Y)$ such that all the cohomology groups up to dimension $r$ is generated by Stiefel-Whitney classes of some vector bundle over $Y$. The integer $r$ is called the \textit{upper characteristic rank} of $Y$. Formally, a vector bundle $\xi$ over a CW complex $Y$, is said to have \textit{characteristic rank} $r$ if $r~(\leq \dim(Y))$ is the largest integer such that all the cohomology classes in $H^i(Y; \bb Z_2), i\leq r$ can be expressed as a polynomial in the Stiefel-Whitney classes of $\xi$. The characteristic rank of a vector bundle 
$\xi$ is denoted by $\mathrm{charrank}(\xi)$.
\begin{definition}
   Let $Y$ be a CW complex. The upper characteristic rank of $Y$, denoted by $\mathrm{ucharrank}(Y)$, is defined to be
   \[\mathrm{ucharrank}(Y) = \max\{\mathrm{charrank}(\xi)|~\xi \textit{~is a bundle over $Y$}\}.\]
\end{definition}
The characteristic rank of a smooth manifold was introduced by Korba\v{s} in \cite{Korbaš2010}. Naolekar and Thakur, in \cite{NT2014}, introduced the characteristic rank of vector bundles over CW complexes. In terms of their definition, for a closed smooth manifold $M$, the characteristic rank $\mathrm{charrank}(M)$ of $M$ is the $\mathrm{charrank}(TM)$, the characteristic rank of the tangent bundle $TM$ of $M$. Recall that the cup length $\mathrm{cup}_R(M)$ of a cohomology ring $H^{*}(M; R)$, where $R$ is the coefficient ring, is the largest $k$ such that there exists $y_1, y_2,\ldots,y_k \in H^{*}(M;R)$ with $y_1 \smile y_2 \smile \cdots \smile y_k\neq 0$. Korba\v{s} showed that if $M$ is a connected smooth manifold of dimension $d$, unoriented cobordant to zero, and $\tilde{H}^k(M; \bb Z_2)$ is the first nontrivial reduced cohomology group, then 
    \[
        \mathrm{cup}_{\bb Z_2}(M)\leq 1+\frac{d-1-\mathrm{charrank(TM)}}{k},
    \]
provided $\mathrm{charrank(TM)}\leq d-2$ (see \cite{Korbaš2010}). He asked what the $\mathrm{charrank}(M)$ for all closed, connected manifolds is. The authors, in \cite{NT2014}, coined the term upper characteristic rank and generalised the above bound for the $\bb Z_2$-cup length (see \cite[Theorem 1.2]{NT2014}).

Although the upper characteristic rank of a CW complex is an important invariant and has implications to the $\bb Z_2$-cup length, Betti number and LS-category (see \cite{KNT2012}, \cite{korbas2014}), the computation of upper characteristic rank is generally a complex problem, and not much seems to be known. In \cite{BK2013}, the authors have derived upper bounds
for the characteristic rank of the oriented Grassmann manifold $\tilde{G}_{n,k}$ if $n$ is odd, with a sharp upper bound for $\mathrm{charrank}(\tilde{G}_{n,2})$, and provided a general upper bound for the characteristic rank of smooth null-cobordant manifolds. In \cite{Korbaš2015}, \cite{PETROVIC2017}, the authors determined the characteristic rank of the canonical vector bundle over the oriented Grassmann manifold $\tilde{G}_{k,n}$ of $k$-dimensional subspaces in $\bb R^n$ for specific values of $k$ and $n$ and determined the $\bb Z_2$-cup-length for these spaces. Balko and Korba\v{s} computed the characteristic rank of smooth null-cobordant manifolds in \cite{BK2013}. In \cite{BK2010}, they have also computed bounds for the characteristic rank of null cobordant manifolds and total spaces of smooth fibre bundles whose fibres are totally non-homologous to zero. For $n\ge 8$, in \cite{BC2020}, Basu and Chakraborty have calculated the upper characteristic rank of $\tilde{G}_{n,3}$ to be the characteristic rank of the oriented tautological bundle over $\tilde{G}_{n,3}$. For the product of spheres, real and complex projective spaces, the Dold manifolds and the stunted projective spaces, the characteristic rank of vector bundles was computed in \cite{NT2014}. In \cite{KNT2012}, the authors have given a complete description of the upper characteristic rank of the Stiefel manifolds $\bb FV_{n,k}$ where $\bb F=\bb R, \bb C$ or $\bb H$, except for few cases where they could provide lower bounds for upper characteristic rank (see Theorem \ref{KNT_main}).   

In this article, we compute the upper characteristic rank of the projective Stiefel manifolds and the flip Stiefel manifolds. Recall that the Stiefel manifold $\bb FV_{n, k}$ is the space of orthonormal-$k$ frames in $\bb F^n$. Note that $\bb FV_{n, k}\subset \prod_k \bb FS^{n-1}$ where $\bb FS^{n-1}$ is the sphere of $\bb F^n$. The Stiefel manifold inherits the subspace topology from the product of spheres in $\bb F^n$. The projective Stiefel manifolds $\bb FX_{n, k}$ is obtained as quotient of $\bb FV_{n, k}$ by the free left action, given by $z(w_1, \dots, w_k)=(zw_1, \dots, zw_k)$, of the group of absolute value 1 elements of $\bb F$. The flip Stiefel manifolds were recently introduced by Basu \textit{et al.} in \cite{BFG2023}. Let $C_n$ be the cyclic group of order $n$ and $C_2= \lra{a}$. The flip Stiefel manifold $FV_{n, 2k}$ is defined to be the quotient of $\bb RV_{n,2k}$ by the action $a(x_1,x_2,\ldots, x_{2k-1}, x_{2k}) =(x_2, x_1,\ldots, x_{2k}, x_{2k-1})$ of $C_2$.

We assume $k>1$ for $\bb RX_{n,k}$ and $2k\geq 2$  for $FV_{n, 2k}$ throughout. Our computations are complete except for a few cases where we could provide bounds for the upper characteristic rank. The main theorem of \S \ref{Real_case} is the following.

Let $N_1=\min\big{\{}j~|~n-k+1\leq j\leq n \text{, and $\binom{n}{j}$ is odd}\big{\}}$ and $N_2= \min\big{\{}j~|~n-2k+1\leq j\leq n \text{ and $\binom{k+j-1}{j}$ is odd}\big{\}}$. Let $X$ be either the projective Stiefel manifold $\bb RX_{n, k}$ or the flip Stiefel manifold $FV_{n, 2k}$, and $c=1$ or $2$ respectively as $X=\bb RX_{n, k}$ or $FV_{n, 2k}$. 
\begin{theorem}\label{main1}
    Let $1<k<n$. The upper characteristic rank of $X$ is as follows:
    
          (a) Let $n - ck\neq 1,2,4$ or $8$. \vspace{-0.375cm}
          \begin{enumerate} 
              \item If $N_c=n-ck+1$, then $\mathrm{ucharrank}(X)\geq n-ck$, the equality holds if $n-ck$ is odd and $n-ck+1$ is not a power of 2, or $n-ck$ is even.
              \item If $N_c\neq n-ck+1$, then $\mathrm{ucharrank}(X)= n-ck-1$.
          \end{enumerate}\vspace{-0.375cm}
          (b) Let $n - ck = 1$. \vspace{-0.375cm}
          \begin{enumerate}
              \item If $n\equiv 2, 3\pmod 4$, then $\mathrm{ucharrank}(\bb RX_{n,k})=2$. 
              \item If $n\equiv 0,2\pmod 4$, then $\mathrm{ucharrank}(FV_{n, 2k})\geq 2$.
              \item The $\mathrm{ucharrank}(X)= 0$ otherwise.
          \end{enumerate}\vspace{-0.375cm}
          (c) Let $n-ck = 2$.\vspace{-0.375cm}
          \begin{enumerate}
              \item If $N_c\neq 4$, then $1\le\mathrm{ucharrank}(X)\le 2$. Further, if $N_c=3$, then $\mathrm{ucharrank}(X)=2$.
              \item If $N_c=4$ then $1\le \mathrm{ucharrank}(X)\le 4$.
          \end{enumerate}\vspace{-0.375cm}
          (d) Let $n-ck = 4$.\vspace{-0.375cm}
          \begin{enumerate}
              \item If $N_c\neq 6$, then $3\leq \mathrm{ucharrank}(X)\leq 4$. Further, if $N_c=5$, then $\mathrm{ucharrank}(X)=4$.
              \item If $N_c=6$, then $3\leq \mathrm{ucharrank}(X)\leq 6$.
          \end{enumerate}\vspace{-0.375cm}
          (e) If $n-ck=8$.  \vspace{-0.375cm}
          \begin{enumerate}
              \item If $N_c\neq 10$, then $7\le\mathrm{ucharrank}(X)\leq 8$, and in particular, for $N_c = 9$, $\mathrm{ucharrank}(X)=8$.
              \item If $N_c =10$, then $7\le\mathrm{ucharrank}(X)\le 10$.
         \end{enumerate}
    \end{theorem}
    In the case of the complex and quaternionic projective Stiefel manifolds, we prove the following in \S\ref{complex}.
\begin{theorem}\label{main2}
 The upper characteristic ranks of the complex and quaternionic projective Stiefel manifolds are as follows:\\
 \[\mathrm{ucharrank}(\bb HX_{n, k})= \begin{cases}
        4(n-k)+6 & \textit{if $\binom{n}{n-k+1}$ is odd,}\\
        4(n-k)+2 & \textit{if $\binom{n}{n-k+1}$ is even.}
    \end{cases}\] 
   And, for $k<n$,\[\mathrm{ucharrank}(\bb CX_{n,k}) = \begin{cases}
        2(n-k)+2 & \textit{if $\binom{n}{n-k+1}$ is odd,}\\
        2(n-k) & \textit{if $\binom{n}{n-k+1}$ is even.}
    \end{cases}\] 
\end{theorem}
 In Section \ref{S^3_index}, we consider the problem of the existence of an $S^3$-equivariant map between quaternionic Stiefel manifolds. One of the useful tools for ruling out equivariant maps between two certain $G$-spaces is the Fadell-Husseini index. Let $G$ be a compact Lie group. A map $f\colon X\to Y$ between two $G$-spaces $X$ and $Y$ is a $G$-equivariant map (or $G$-map), if 
 \[f(g\cdot x)=g\cdot f(x); \forall g\in G, \forall x\in X.\] 
 The existence of equivariant maps between two $G$-spaces are related to Borsuk-Ulam type theorems (see \cite{FH1988}), and hence it is an interesting problem to look for such $G$-equivariant maps between $G$-spaces. Computation for the necessary and sufficient conditions of existence of certain $G$-maps between real and complex Stiefel manifolds using the Fadell-Husseini index has been described in \cite{Ha05}, \cite{Pe13}, \cite{BK21}.  Note that the 3-dimensional sphere $S^3$ can be identified with the group of the unit quaternions in $\bb H$. In Section \ref{S^3_index}, after recalling the Fadell-Husseini index (see \cite{FH1988}) associated with a space $X$, we use it for computing some necessary conditions for the existence of an $S^3$-equivariant map between the quaternionic Stiefel manifolds and certain $(4n-1)$-dimensional spheres. Similar problems have been studied by authors in \cite{PETROVIC2017} and \cite{BK21} for the complex and real Stiefel manifolds. Our main result of Section \ref{S^3_index} is as follows. 
 \begin{theorem}\label{main3}
 (a) If there exists an $S^3$-map between $\bb HV_{n,k}\to \bb HV_{m,l}$,  then $n-k\leq m-l$. Moreover if $n-k=m-l$, $\binom{n}{n-k+1}\vert \binom{m}{m-l+1}$.\\
  (b) There exists an $S^3$-map $f:~Sp(n)\to Sp(m)$ if and only if $n~|~m$.\\
    (c) If there exists an $S^3$-map between $ S^{4n-1}\to \bb HV_{m,l}$, then $n\leq m-l+1$.\\
    (d) If there exists an $S^3$-map $\bb HV_{n,k}\to  S^{4m-1}$, then $m\geq n-k+1$.
 \end{theorem}

The upper characteristic rank of the real projective Stiefel manifolds $\bb RX_{n, k}$ was independently obtained by the author in \cite{dasgupta2024} for some particular values of $n-k$, in fact for $n-k=5, 6$ or $\ge 9$. The article also deals with the complex case. The present article, however, concerns all possible values of $n-k$, and our approach is different from that of \cite{dasgupta2024}.
\numberwithin{equation}{section}
\section{Real Projective Stiefel Manifolds}\label{Real_case}
 We prove Theorem \ref{main1} in this section. To begin with, let us briefly recall the mod two cohomology algebra of the real projective Stiefel manifolds and flip Stiefel manifolds. Let $X$ and $c$ be as in Theorem \ref{main1}. Let $\xi$ be the line bundle associated to the double covering $\bb RV_{n, ck}\to X$ and $\ga$ be the Hopf line bundle over $BC_2$. Consider the fibration
                         \[ \bb RV_{n,ck}\to BO_{n-ck}\to BO_n. \]
 The classifying map $BC_2\to BO_n$ of the bundle $n\ga$ induces a corresponding fibration over $BC_2$ 
        \begin{equation}\label{f1}
        \bb RV_{n,ck}\xrightarrow{i} Y_{n, ck}\xrightarrow{p} BC_2,
        \end{equation}
where the total space $Y_{n, ck}$ is homotopy equivalent to the space $X$ (see \cite{GH68}), and henceforth we do not differentiate between $Y_{n, ck}$ and $X$.
Borel, in \cite{Bo53}, showed that the cohomology algebra of $\bb RV_{n,\ell}$ with $\bb Z_2$ coefficients is
                \[H^*(\bb RV_{n,\ell}; \bb Z_2) = V(z_{n-\ell},\dots, z_{n-1}).\]
Here $V(z_{n-\ell},\dots, z_{n-1})$ denotes a graded commutative, associative algebra with unit over $\bb Z_2$ and the classes $z_j$ have dimensions $j$, for $n-\ell\leq j\leq n-1$. The elements $1$ and $z_{i_1}z_{i_2}\cdots z_{i_p}; z-\ell\leq i_1<i_2<\cdots<i_p\leq n-1$ form a basis for $V(z_{n-\ell},\dots, z_{n-1})$ over $\bb Z_2$, satisfying $z_j^2=z_{2j}$ if $2j\leq n-1$ and $z_j^2=0$ otherwise. 
The action of the Steenrod algebra on $H^*(\bb RV_{n, \ell};\bb Z_2)$ is given by (see \cite{Bo53})
\begin{equation}\label{sqV}
Sq^j(z_q)=\binom{q}{j} z_{q+j} ~~~~~\textit{for $j\leq q.$} 
\end{equation} 
Let $X, c$ and $N_c$ be as defined before Theorem $\ref{main1}$. The mod two cohomology algebra of $X$ has the following description. 
\begin{theorem}\label{rpsm_cohomology}{\em (\cite[Theorem 2.1]{GH68}; \cite[Theorem 3.5]{BFG2023})} 
Suppose $ck<n$.  Let $x$ be the generator of the algebra $H^*(BC_2; \bb Z_2)$ of dimension 1. Then 
    \[H^*(X; \bb Z_2) = \bb Z_2[y]/\lra{y^{N_c}}\otimes V(y_{n-ck},\dots,\hat{y}_{N_c-1},\dots y_{n-1}),\]
    where $y=p^*(x)$ and $z_j=i^*({y_{j}})$ and $n-ck\leq j\leq n-1$.\hfill$\square$
\end{theorem}

Note that the class $y$ is the pullback of the class $x$, and hence $i^*(y)=0$.
Before proving Theorem \ref{main1}, we recall the upper characteristic rank of the projective Stiefel manifolds, which we will use on several occasions while proving Theorem \ref{main1}.
 \begin{theorem}{\em (\cite[Theorem 1.1]{KNT2012})}\label{KNT_main}
     Let $1<k<n$ when $\bb F=\bb R$ and $1<k\leq n$ when $\bb F=\bb C, \bb H$.\\
     (a) If $\bb F=\bb R$, then
     \[\mathrm{ucharrank}(\bb FV_{n, k})=\begin{cases}
         n-k-1 & \textit{if $n-k\neq 1, 2, 4, 8,$}\\
         2     & \textit{if $n-k= 1$ and $n\geq 4$,}\\
         2     & \textit{if $n-k= 2$,}\\
         4     & \textit{if $n-k= 4$ and $k=2$.}
     \end{cases}\]
     (b) If $k>2$ and $n-k=4$, then $\mathrm{ucharrank}(\bb RV_{n, k})\leq 4$.\\
     (c) If $n-k=8$, then $\mathrm{ucharrank}(\bb RV_{n, k})\leq 8$.\\
     (d) If $\bb F=\bb C$, then
     \[\mathrm{ucharrank}(\bb FV_{n, k})=\begin{cases}
         2     & \textit{if $k=n$,}\\
         2(n-k)     & \textit{if $k<n$.}\\
     \end{cases}\]
     (e) If $\bb F=\bb H$, then $\mathrm{ucharrank}(\bb FV_{n, k})= 4(n-k)+2$.\\
 \end{theorem}
We are now ready to prove Theorem \ref{main1}. We break the proof into the following five lemmata.
\begin{lemma}\label{(a)}
    Let $n - ck\neq 1,2,4$ or $8$, and $X$ be either $\bb RX_{n,k}$ or $FV_{n,2k}$. \vspace{-0.375cm}
          \begin{enumerate} 
              \item If $N_c=n-ck+1$, then $\mathrm{ucharrank}(X)\geq n-ck$, the equality holds if $n-ck$ is odd and $n-ck+1$ is not a power of 2, or $n-ck$ is even.
              \item If $N_c\neq n-ck+1$ then $\mathrm{ucharrank}(X)= n-ck-1$.
          \end{enumerate}
\end{lemma}
\begin{proof}
    Let us assume that  $n-ck\neq 1, 2, 4$ or $8$. If $N_c=n-ck+1$, there is no degree $N_c-1$ class in the cohomology ring $H^*(X; \bb Z_2)$ which is obtained from $H^*(V_{n,ck}; \bb Z_2)$ under $i^*$, that is $y_{n-ck}=0$. Also, all the cohomology classes in $H^*(X; \bb Z_2)$ of degree less than $n-ck+1$ are obtained as the Stiefel-Whitney classes of $p^*\gamma$. Therefore, $\mathrm{ucharrank}(X)\ge n-ck$. 
    
    Now, let $n-ck$ be odd, and $n-ck+1$ be not a power of 2. We show that $\mathrm{ucharrank}(X)\le n-ck$, by proving that there cannot be any vector bundle $\alpha$ and polynomials $P_j; 1\leq j\leq N_c$ and $Q_{n-ck+1}$ in Stiefel-Whitney classes of $\alpha$ such that $P_j(w_1(\alpha), \ldots, w_j(\alpha))=y^j$, and $Q_{n-ck+1}(w_{1}(\alpha), \ldots, w_{n-ck+1}(\alpha)) = y_{n-ck+1}$. To the contrary, if such a vector bundle and such polynomials exist, let $\beta = i^*\alpha$. Then, $z_{n-ck+1} = i^*(y_{n-ck+1})=Q_{n-ck+1}(w_1(\beta), \ldots, w_{n-ck+1}(\beta))$. Since there is no non-trivial cohomology group of $V_{n, ck}$ in dimension less than $n-ck$ and $n-ck>1$, we must have $Q_{n-ck+1}(\beta) = a\cdot w_{n-ck+1}(\beta)$, where $a$ is a non-zero constant and hence $a=1$. Therefore, $w_{n-ck+1}(\beta)=z_{n-ck+1}$. We show that $\beta$ has the first non-trivial Stiefel-Whitney class in dimension $n-ck+1$. If not, the only possibility is that $w_{n-ck}(\beta)\neq 0$. Since $w(\beta)\neq 1$, it is well-known that the smallest positive $m$ such that  $w_m(\beta)\neq 0$ must be a power of two. But, $n-ck(\neq 1)$ is odd and hence, no vector bundle can have the first non-trivial Stiefel-Whitney class in dimension $n-ck$. Therefore, the first non-trivial Stiefel-Whitney class of $\beta$ occurs in dimension $n-ck+1$, as claimed. Now, since $n-ck+1$ is not a power of 2, no vector bundle can have the first non-trivial Stiefel-Whitney class in dimension $n-ck+1$, that is $z_{n-ck+1}=w_{n-ck+1}(\beta)$ must be 0, a contradiction. Therefore, $\mathrm{ucharrank}(X)$ must be $n-ck$. 
    
     Let $n-ck$ be even. We show that if $\mathrm{ucharrank}(X)> n-ck$, then there must be a vector bundle $\beta$, over $V_{n, ck}$ with $w_{n-ck}(\beta)\neq 0$ which will contradict that $\mathrm{ucharrank}(\bb RV_{n, ck}=n-ck+1$. Following exactly the same arguments as above, there is a vector bundle $\beta$ over $V_{n,k}$ such that $w_{n-ck+1}(\beta)\neq 0$. If $w_{n-ck}(\beta) = 0$, then by Wu's formula, $0=Sq^1(w_{n-ck}(\beta))=(n-ck-1)w_{n-ck+1}(\beta)=w_{n-ck+1}(\beta)$, a contradiction. Therefore, $w_{n-ck}(\beta)$ must be $z_{n-ck}\neq 0$, as claimed. Hence, we have $\mathrm{ucharrank}(X)=n-ck$. This proves (1).
    
    

     Let us now assume $N_c> n-ck+1$. As above, if there exist a vector bundle $\alpha$ over $X$ and polynomials $P_j; 1\leq j\leq n-ck$ such that $P_j(w_1(\alpha), \ldots, w_{j}(\alpha))=y^j$ and $Q_{n-ck}(w_{1}(\alpha), \ldots, w_{n-ck}(\alpha)) = y_{n-ck}$, then $z_{n-ck} = i^*(y_{n-ck})=Q_{n-ck}(w_{1}(\beta), \ldots, w_{n-ck}(\beta))$. This implies that $\mathrm{ucharrank}(\bb RV_{n, ck})\ge n-ck$, which is a contradiction as $\mathrm{ucharrank}(\bb RV_{n, ck}) = n-ck-1$ by Theorem \ref{KNT_main}. Therefore, $\mathrm{ucharrank}(X)\leq  n-ck-1$. Also, all the cohomology classes in $H^*(X; \bb Z_2)$ of degree less than $n-ck$ are obtained as the Stiefel-Whitney classes of $p^*\gamma$. Hence, $\mathrm{ucharrank}(X)= n-ck-1$ in this case. This completes the proof.
     \end{proof}

    

\begin{lemma}\label{(b)}
      Let $n - ck = 1$.\vspace{-0.375cm}
          \begin{enumerate}
              \item If $n\equiv 2, 3\pmod 4$ then $\mathrm{ucharrank}(\bb RX_{n,k})=2$. 
              \item If $n\equiv 0,2\pmod 4$ then $\mathrm{ucharrank}(FV_{n, 2k})\geq 2$.
              \item The $\mathrm{ucharrank}(X)= 0$ otherwise.
          \end{enumerate}
\end{lemma}
     \begin{proof}
    Let $n-ck$ be $1$. Then $\mathrm{ucharrank}(\bb RV_{n, ck})=2$. If $N_c\neq 2$, then $H^1(X; \bb Z_2)\cong \bb Z_2\oplus \bb Z_2$, and hence the upper characteristic rank must be zero. Therefore, we assume $N_c= 2$, that is,  the class $y_1$ is trivial in the cohomology of $X$. Note that $N_c=2$ if and only if $n\equiv 2, 3\pmod 4$ for $\bb RX_{n,k}$ and $n\equiv 0, 2\pmod 4$ for $FV_{n, 2k}$.  We write $\bb RV_{ck+1, ck}$ for $\bb RV_{n,ck}$.
    
    Now, for $N_c=2$, as $H^1(X;\bb Z_2)$ is generated by the Stiefel-Whitney class of $p^*\ga$, the $\mathrm{ucharrank}(X)\geq 1$. We show that the homomorphism $\rho_2\colon H^2(X;\bb Z)\to H^2(X;\bb Z_2)$ is surjective. Since the elements of $H^2(X; \bb Z)$ are in one-to-one correspondence with the Chern class of complex line bundles over $X$, this will prove that $y_2$ can be obtained as the second Stefel-Whitney class of some orientable 2-plane bundle $\beta$ over $X$. Hence, all the classes in $H^2(X; \bb Z_2)$ can be obtained as polynomials in the Stiefel-Whitney classes of the bundle $p^*(\ga)\oplus \beta$. Therefore, $\mathrm{ucharrank}(X)\geq 2$.
    
    Consider the homotopy long exact sequence induced by $\bb RV_{ck+1,ck}\xrightarrow{i} X\xrightarrow{p} \bb RP^{\infty}$. Since $\pi_j(\bb RP^\infty)=0; j>1$, hence $\pi_j(X)\cong \pi_j(\bb RV_{ck+1, ck}); j>1$. Since, $\pi_1(\bb RV_{ck+1, ck})=\bb Z_2$ we have the following exact sequence
                \[0\to \bb Z_2\to \pi_1(X)\to \bb Z_2\to 0.\]
                        
    First, we show that $\pi_1(X)=\bb Z_4 $. If $\pi_1(X)=\bb Z_2\oplus \bb Z_2$, then $H^1(X; \bb Z)=\bb Z_2\oplus \bb Z_2$, which contradicts that $H^1(X; \bb Z_2)$ is $\bb Z_2$ (see Theorem \ref{rpsm_cohomology}). Therefore, $\pi_1(X) = \bb Z_4 $. Since, $H_2(B\bb Z_4, \bb Z)=0$ and $\pi_2(X)\cong \pi_2(\bb RV_{ck+1, ck})\cong 0$, the right exact sequence $\pi_2(X)\to H_2(X; \bb Z)\to H_2(\pi_1(X), \bb Z)\to 0$ (see \cite[(0.1)]{brown2012cohomology}) yields $H_2(X; \bb Z)\cong \pi_2(X)\cong 0$, and hence $H^2(X; \bb Z)$ is $\bb Z_4 $. Therefore, the Bockstein long exact sequence
\[ 0\xrightarrow{} H^1(X; \bb Z_2)\xrightarrow{\delta} H^2(X; \bb Z)\to H^2(X; \bb Z)\xrightarrow{\rho_2} H^2(X; \bb Z_2)\rightarrow\cdots\]

becomes
            \[0\to\bb Z_2\to \bb Z_4 \to \bb Z_4 \to\bb Z_2\to\cdots.\]

Since $\delta$ is injective, $\rho_2$ must be surjective.

Finally, we show that $\mathrm{ucharrank}(\bb RX_{n+1, n})< 3$. Already we have a bundle, namely $p^*(\ga)\oplus \beta$, such that $w_3(p^*(\ga)\oplus \beta)=yy_2$. We need to show that there cannot be any vector bundle $\alpha$ and polynomials $P_2$ and $P_3$ such that 
\begin{align*}
    w_1(\alpha) & = y,\\
    P_2(w_1(\alpha), w_2(\alpha))&=y_2,\\
    P_3(w_1(\alpha), w_2(\alpha), w_3(\alpha))&= y_3.
\end{align*} 
Since, $y^2=0$, we get $P_2(w_1(\alpha), w_2(\alpha))=w_2(\alpha)=y_2$, and hence, $w_2(i^*\alpha)=z_2$ and $P_3(w_1(i^*\alpha), z_2, w_3(i^*(\alpha))=z_3$. Since $\mathrm{ucharrank}(\bb RV_{n, ck})=2$, we must have $w_1(i^*\alpha)=0$. Now, using Wu's formula and $Sq^1(z_2)= 2z_3$, we have $w_3(i^*\alpha) = w_1(i^*\alpha)w_2(i^*\alpha)+w_3(i^*\alpha)= Sq^1(w_2(i^*\alpha))=Sq^1(z_2)=2z_3=0$.
 Therefore, $P_3(0, w_2(i^*\alpha), 0)=z_3$, which is impossible as the left side is even-dimensional while the right side is odd. Therefore, $\mathrm{ucharrank}(\bb RX_{n+1, n}) = 2$.
\end{proof}

\begin{lemma}\label{(c)}
    Let $n-ck = 2$.\vspace{-0.375cm}
          \begin{enumerate}
              \item If $N_c\neq 4$, then $1\le\mathrm{ucharrank}(X)\le 2$. Further, if $N_c=3$ then $\mathrm{ucharrank}(X)=2$.
              \item If $N_c=4$, then $1\le \mathrm{ucharrank}(X)\le 4$.
          \end{enumerate}
\end{lemma}
\begin{proof}
    To begin with, note that if $n-ck=2$, then we have $\mathrm{ucharrank}(\bb RV_{n, ck})=2$ and $N_c\ge 3$. Also, as $n-k=2$ and the only cohomology class in $H^1(X; \bb Z_2)$ is $w_1(p^*\ga)$, the $\mathrm{ucharrank}(X)$ is always at least one. 
    
    Let us assume $N_c\neq 4$. We show $\mathrm{ucharrank}(X)<3$. If not, there must be a bundle $\alpha$ over $X$ with $w_3(\alpha) = y_3$. This is because it is clear from Theorem \ref{rpsm_cohomology} that $y_3$ is not a polynomial combination of $y$ and $y_2$.  Then, $w_3(i^*\alpha)=z_3$, and as $\mathrm{ucharrank}(\bb RV_{n, ck}) =2$, $w_2(i^*\alpha)=0$. Therefore, $0=sq^1(w_2(i^*\alpha))=w_3(\alpha)=z_3$, a contradiction. Hence, if $N_c\neq 4$, then $\mathrm{ucharrank}(X)\le2$.

    If $N_c=3$, then $H^i(X; \bb Z_2)\cong \bb Z_2\lra{y^i}; i=1, 2$ and, therefore, $\mathrm{ucharrank}(X)\geq 2$. This completes the proof of $(1)$. 

Now, let $N_c =4$. If the $\mathrm{ucharrank}(X)\geq 5$, it follows from Theorem \ref{rpsm_cohomology} that there must be a vector bundle $\alpha$ over $X$ such that $w_4(\alpha)= y_4$ and $w_5(\alpha)=y_5$. Then $w_5(i^*\alpha)= z_5$. Now,  by the relations \ref{sqV},
\[z_5=w_5(i^*\alpha)=w_1(i^*\alpha)w_4(i^*\alpha)+sq^1(w_4(i^*\alpha))=sq^1(z_4)=0.\] This is a contradiction, and hence $\mathrm{ucharrank}(X)\leq 4$. This completes the proof.
\end{proof}
\begin{corollary}
    If $n-k=2$, then $\mathrm{ucharrank}(\bb RX_{n, k})=2$ for $n\equiv 3\pmod 4$ and $\mathrm{ucharrank}(FV_{n, k})=2$ if $n\equiv 0\pmod 4$.
\end{corollary}
\begin{proof}
     Since $N_c=3$ if $n=4m+3$ for $c=1$ and $n=4m$ for $c=2$, the corollary follows from Lemma \ref{(c)} (1).
\end{proof}
\begin{lemma}\label{(d)}
    Let $n-ck = 4$.\vspace{-0.375cm}
          \begin{enumerate}
              \item If $N_c\neq 6$, then $3\leq \mathrm{ucharrank}(X)\leq 4$. Further, if $N_c=5$, then $\mathrm{ucharrank}(X)=4$.
              \item If $N_c=6$, then $3\leq \mathrm{ucharrank}(X)\leq 6$.
          \end{enumerate}
\end{lemma}
\begin{proof}
For $n-ck=4$, since $N_c\geq 5$, therefore $H^i(X)\cong \bb Z_2\lra{y^i}; i=1, 2, 3$, and hence $\mathrm{ucharrank}(X)\geq 3$.

We show that if $N_c\neq 6$, then $\mathrm{ucharrank}(X)\leq 4$. If possible, let there exists vector bundles $\alpha$ over $X$ and polynomial $P$, such that $P(w_{1}(\alpha), \ldots, w_{4}(\alpha)) = y_5$. Clearly $w_1(i^*\alpha)=w_2(i^*\alpha)= w_3(i^*\alpha) =0$. Then, $P(0,0,0,w_4(i^*\alpha), w_{5}(i^*\alpha)) = z_5$. Since $\mathrm{ucharrank}(\bb RV_{n, ck})\leq 
4$ and $z_5$ is expressible as a polynomial in the Stiefel-Whitney classes of $i^*\alpha$, we must have $w_4(i^*\alpha)=0$. Now, by Wu's formula, \[0=Sq^1(w_4(i^*\alpha))=w_1(i^*\alpha)w_4(i^*\alpha)+w_{5}(i^*\alpha)=w_5(i^*\alpha).\] Therefore, $P(0, 0, 0,0, 0)=z_5$, which is not possible. Hence, $\mathrm{ucharrank}(X) \leq 4$.

Now $N_c=5$, when $n\equiv 1,3 \pmod4$ for $\bb RX_{n+4, n}$ and $n\equiv 2 \pmod4$ for $FV_{n,2k}$. Since $y_4$ vanishes for $N_c=5$, the $\mathrm{ucharrank}(X)$ is exactly $4$. This proves (1).

Let $N_c=6$. If $\mathrm{ucharrank}(X)\geq 7$, then there must be a vector bundle $\alpha$ over $X$ such that $w_6(\alpha)=y_6$ and $w_7(\alpha)=y_7$. The exact same arguments as in Lemma \ref{(c)} (2) show that the existence of such a bundle implies $z_7=0$, a contradiction. Hence $\mathrm{ucharrank}(X)\leq 6$. Hence the lemma.
\end{proof}
Note that $N_c = 5$ if $n\equiv 1,3\pmod 4$ for $c=1$ and $n\equiv 2\pmod 4$ for $c=2\pmod 4$. Hence, the next corollary follows.
\begin{corollary}
    If $n-k=4$, then $\mathrm{ucharrank}(\bb RX_{n, k})=4$ for $n\equiv 1,3\pmod 4$ and $\mathrm{ucharrank}(FV_{n, 2k})=4$ if $n\equiv 2\pmod 4$. 
\end{corollary}
    Note that $N_c = 5$ if $n\equiv 1,3\pmod 4$ for $c=1$ and $n\equiv 2\pmod 4$ for $c=2\pmod 4$. Therefore, by Lemma \ref{(d)}, $\mathrm{ucharrank}(\bb RX_{n, k})=4$ for $n\equiv 1,3\pmod 4$ and $\mathrm{ucharrank}(FV_{n, 2k})=4$ if $n\equiv 2\pmod 4$. 
\begin{lemma}\label{(e)}
    If $n-ck=8$.  \vspace{-0.375cm}
    \begin{enumerate}
        \item If $N_c\neq 10$, then $7\le\mathrm{ucharrank}(X)\leq 8$, and in particular, for $N_c = 9$, $\mathrm{ucharrank}(X)=8$.
        \item If $N_c =10$, then $7\le\mathrm{ucharrank}(X)\le 10$.
    \end{enumerate}
\end{lemma}
\begin{proof}
   Since the cohomology groups up to dimension 7 can be generated by the Stiefel-Whitney classes of $p^*(\ga)$, hence $\mathrm{ucharrank}(X)\geq 7$.
   
   Let $N_c\neq 10$. If possible, let $\mathrm{ucharrank}(X)\geq 9$. Then, there is a vector bundle $\alpha$ over $X$ and polynomials $P_j; 1\leq j\leq N_c$ and $Q_j; n-ck\leq j\leq \mathrm{ucharrank}(X)$ such that $P_j(w_1(\alpha), \ldots, w_j(\alpha))=y^j$ and $Q_j(w_1(\alpha), \ldots, w_j(\alpha))=y_j$. Now, $Q_9(w_1(i^*\alpha), \ldots, w_9(i^*\alpha))=z_9$, where $w_j(i^*\alpha)= 0$ for $1\leq j\leq 7$. Again, $Sq^1(w_8(i^*\alpha))=w_9(i^*\alpha)$ by Wu's formula. Hence, if $w_8(i^*\alpha)=0$, then $w_9(i^*\alpha)=0$. On the other hand, if $w_8(i^*\alpha)=z_8$, then by relations \ref{sqV}, $Sq^1(w_8(i^*\alpha))=8z_9=0$. Therefore, in any case, whatever $w_8(i^*\alpha)$ be, $w_9(i^*\alpha)=0$. Hence, $Q_9(0,\ldots, 0, w_8(i^*\alpha), 0) =z_9$, which is again impossible. Therefore, such bundle $\alpha$ do not exist and $\mathrm{ucharrank}(X)\leq 8$.
   
   Now, if $N_c=9$, the class $y_8$ is not present in the cohomology of $X$ and $H^j(X; \bb Z_2)$ is generated by Stiefel-Whitney classes of $p^*\ga$. Therefore, $\mathrm{ucharrank}(X)=8$ in this case. This proves $(1)$.
   
   The proof of $(2)$ is parallel to that of Lemma \ref{(c)} (2). This completes the proof of the lemma.
\end{proof}
The Theorem \ref{main1} now follows from Lemma \ref{(a)}, \ref{(b)}, \ref{(c)}, \ref{(d)}, and \ref{(e)}.
\section{Complex and Quaternionic Projective Stiefel Manifold}\label{complex}
We denote the complex and quaternionic Stiefel manifolds respectively by $\bb CV_{n,k}$ and $\bb HV_{n,k}$ and the corresponding projective Stiefel manifolds by $\bb CX_{n,k}$ and $\bb HX_{n,k}$. The mod 2 cohomology algebra of $\bb CV_{n,k}$ and $\bb HV_{n, k}$ are well known and have the following descriptions (See \cite{Bo53}).

\begin{equation*}
    H^*(\bb CV_{n,k}; \bb Z_2)\cong \Lambda_{\bb Z_2}(z'_{n-k+1}, z'_{n-k+2}, \cdots, z'_{n}),\quad   |z'_j|=2j-1,
\end{equation*}
\begin{equation*}
    H^*(\bb HV_{n,k}; \bb Z_2)\cong \Lambda_{\bb Z_2}(z''_{n-k+1}, z''_{n-k+2}, \ldots, z''_{n}),~~~  |z''_j|=4j-1.
\end{equation*}

We have the fibrations, which are the complex and the quaternionic analogue of \eqref{f1}

\begin{equation*}
    \bb CV_{n,k}\xrightarrow{i'} \bb CX_{n,k}\xrightarrow{p'} BS^1, 
\end{equation*}  
\begin{equation}\label{f3}
    \bb HV_{n,k}\xrightarrow{i''} \bb HX_{n,k}\xrightarrow{p''} BSp(1), 
\end{equation}

which determine the cohomology algebra of complex and quaternionic projective Stiefel manifolds.
\begin{theorem}{\em (\cite[Theorem A]{Ru69},\cite[Theorem 5]{Zhu-Pop2022})}
Let $k<n$ and $N=\min\{j~|~n-k+1\leq j< n$, such that $ \binom{n}{j} \not\equiv 0 \pmod 2\}$. Let $x'$ be the generator of the algebra $H^*(BS^1; \bb Z_2)$ of dimension $2$ and $x''$ be the generator of the algebra $H^*(BSp(1); \bb Z_2)$ of dimension $4$. Then 
\[H^*(\bb  CX_{n,k}; \bb Z_2)\cong \bb Z_2[y']/\lra{(y')^N}\otimes \Lambda_{\bb Z_{2}}(y'_{n-k+1},\cdots, \hat{y'}_{N}, \cdots, y'_{n}),\quad   |y'_j|=2j-1,\]
\[H^*(\bb  HX_{n,k}; \bb Z_2)\cong \bb Z_2[y'']/\lra{(y'')^N}\otimes \Lambda_{\bb Z_{2}}(y''_{n-k+1}, \cdots,\hat{y''}_{N},\cdots,y''_{n}),\quad   |y_j''|=4j-1,\]

where $(p')^*x'=y', (p'')^*x''=y''$ and $(i')^*y'_j=z'_j, (i'')^*y''_j=z''_j$.
\end{theorem}

 We will give proof for the mod two cohomology algebra of the quaternionic projective Stiefel manifold. The proof is similar and along the lines of the real case as given in \cite{GH68}. We recall some general facts about the universal quaternionic bundle and its associated characteristic classes. The integral cohomology algebra of the classifying space of the symplectic group $Sp(n)$ is the polynomial algebra generated by the universal symplectic classes $k_i$, of degree $4i; 1\le i\le n$, associated with the universal $n$-dimensional bundle over $BSp(n)$, that is
     \[
         H^*(BSp(n);\bb Z)=\bb Z[k_1,...k_n].
     \]
  Since the$\mod 2$ reduction of the symplectic classes are the Stiefel-Whitney classes of the underlying real bundle, we have 
  \[H^*(BSp(n);\bb Z_2)=\bb Z_2[w_4,...w_{4n}].\] 
 
\begin{theorem}{\em (\cite[Theorem 5]{Zhu-Pop2022})}\label{Quaternion_cohomology}
 Let $k<n$ and $N=\min\{j~|~n-k+1\leq j< n$, \text{such that} $ \binom{n}{j} \not\equiv 0 \pmod 2\}$. Let $\omega$ be the generator of the algebra $H^*(BSp(1); \bb Z_2)$. Then 
 \[H^*(\bb  HX_{n,k}; \bb Z_2)\cong \bb Z_2[y]/\lra{y^N}\otimes \Lambda_{\bb Z_{2}}(y_{n-k+1}, \cdots,\hat{y}_{N},\cdots,y_{n}),\quad   |y_i|=4i-1,\] where $p^*(\omega)=y$ and $i^*(y_i)=z_i$.
 \end{theorem}

 \begin{proof}
 Consider the pullback square of fibre bundles associated with \eqref{f3},
 \[
 \xymatrix{
 \bb HV_{n,k}\ar@{=}[rr]\ar[d] & &\bb HV_{n,k}\ar[d]\\
 \bb HX_{n,k}\ar[rr] \ar[d] & &BSp(n-k)\ar[d]\\
  BSp(1)\ar[rr]^{i} & &BSp(n).
 }
 \]
 Here, the map $i$ classifies the bundle $\oplus_{n}\gamma$ where $\gamma$ is the canonical quaternionic line bundle over $BSp(1)$. In the Serre spectral sequence of the right hand fibration  $d_{4i}(z_i)=w_{4i}$ and
 \begin{align*}
     w(\oplus_{n}\gamma)&=(1+\omega)^n\\
                &=\sum_{i=0}^{i=n}\binom{n}{i}\omega^i.
 \end{align*}
 Here $\omega$ is the pullback of the class $w_4$ under the map $i^*$. This implies \[d_{4i}=\binom{n}{i}\omega^i.\] So the first non-zero differential will appear at $E_N$ and all the non-zero power of $\omega$ will vanish on $E_{n+1}$ page by multiplicative property of the Serre spectral sequence. Hence 
\begin{align*}
     E_2&=E_N\\
     E_{N+1}&=E_{\infty}=\bb Z_2[y]/\lra{y^N}\otimes \Lambda_{\bb Z_{2}}(y_{n-k+1}, \cdots, y_{n}).
\end{align*}
 Since we are in $\pmod 2$ coefficient, there is no extension problem. Hence, the proof is complete.
 \end{proof}

We denote the canonical line bundles over $BS^1$ and $BSp(1)$ by $\ga'$ and $\ga''$ respectively.  We are now ready to prove Theorem \ref{main2}.

\begin{proof}[Proof of Theorem \ref{main2}]
    We prove the complex case; the quaternionic case is similar. Recall that for $k<n$ $\mathrm{ucharrank}(\bb CV_{n,k})=2(n-k)$.

    Let $\binom{n}{n-k+1}$ be even, that is, $N\neq n-k+1$. If $\mathrm{ucharrank}(\bb CX_{n,k})\geq 2(n-k)+1$, then there is a vector bundle $\alpha$ over $\bb CX_{n,k}$ with $w_{2(n-k)+1}(\alpha)=y'_{n-k+1}$ and polynomials $P_j; 1\le j\le N-1$ such that $P_j(w_1(\alpha), \ldots, w_j(\alpha))=(y')^j$. Then $w_{2(n-k)+1}((i')^*\alpha)=z_{n-k+1}$. But, since $\bb CV_{n,k}$ has first nontrivial cohomology in dimension $2(n-k)+1$, which is not a power of $2$, $(i')^*\alpha$ cannot have first non-trivial Stiefel-Whitney class in dimension $2(n-k)+1$. Therefore, such a vector bundle $\alpha$ cannot exist and hence, $\mathrm{ucharrank}(\bb CX_{n,k})\leq 2(n-k)$. Now, since $\mathrm{charrank}((p')^*\gamma')=2(n-k)$ the upper characteristic rank of $\bb CX_{n,k}$ is $2(n-k)$. 
    
    If $\binom{n}{n-k+1}$ be odd, that is, $N=n-k+1$, then $H^{2(n-k)+1}(\bb CX_{n,k}; \bb Z_2)=0=H^{2(n-k)+2}(\bb CX_{n,k}; \bb Z_2)$. If $\mathrm{ucharrank}(\bb CX_{n, k})\ge 2(n-k)+3$, the bundle $\beta$ with $w_{2(n-k)+3}=y'_{n-k+2}$ pulls back to a bundle $\alpha$ over $\bb CV_{n, k}$ whose first Stiefel-Whitney class appear in degree $2(n-k)+1$ or $2(n-k)+3$, which is impossible. Therefore, $\mathrm{ucharrank}(\bb CX_{n, k})=2(n-k)+2$ in this case.
\end{proof}

\section{Cup Lengths}\label{cup_length}
We conclude with the following observation about the $\bb Z_2$-cup length of $\bb RX_{n, k}$. To begin with, let us recall the following theorem from \cite{NT2014}.
\begin{theorem}{\em \cite[Theorem 1.2]{NT2014}}\label{TN}
    Let $Y$ be a connected closed smooth manifold of dimension $d$ and $\xi$ be a vector bundle over $Y$ satisfying the following: there exists $j, j\leq charrank(\xi)$, such that every monomial $w_{i_1}(\xi)\ldots w_{i_r}(\xi); 0\leq i_p\leq j$, in dimension $d$ vanishes. Then, 
                        \[\mathrm{cup}(Y)\leq 1+\frac{d-j-1}{r_y}.\]
    Here $\tilde{H}^{r_y}(Y; \bb Z_2)$ is the first non-trivial reduced cohomology group of $Y$.
\end{theorem}
\begin{theorem}
    Let $n\neq k$, the dimension of $\bb RX_{n, k}$ is $d$, and $\binom{n}{n-k+1}$ is odd. Let $k>1$, then $\mathrm{cup}(X_{n,k})\leq d-N$.
\end{theorem}
\begin{proof}
    Let us consider the vector bundle $\xi$ over $\bb RX_{n, k}$ associated to the double covering $\bb RV_{n, k}\to \bb RX_{n, k}$. Then, $w_1(\xi)=y$ and $w_{i_1}(\xi)\ldots w_{i_r}(\xi)=y^{\sum_{p=1}^r i_p}$, which vanishes if and only if $\sum_{p=1}^r i_p\geq N_1$. Hence, for $d\geq N_1$ all the monomials $w_{i_1}(\xi)\ldots w_{i_r}(\xi); 0\leq i_p\leq N_1$ in dimension $d$ vanish. Now, since $\binom{n}{n-k+1}$ is odd, $N_1=n-k+1$. Recall that $d= nk-\frac{1}{2}k(k+1)$. Since $n>k$ and $k\neq 0$, we have $n\geq 2$ and therefore, $n-\frac{1}{2}k\ge n-\frac{n}{2}=\frac{n}{2}\ge 1$. It follows that $d\geq N_1$ hold if $k>1$. Therefore, we may take $j=N_1$.

    Now, since $n-k+1\geq 2, \tilde{H}^1(\bb RX_{n, k}; \bb Z_2)\neq 0$ and hence, $r_x=1$. Therefore, the theorem follows from Theorem \ref{TN}.
\end{proof}
\begin{remark}
    For $k>1, N_2=n-2k+1$ and $N=n-k+1$, the same set of arguments applies to the flip Stiefel manifolds, complex and quaternionic projective Stiefel manifolds, providing upper bounds $d-N_2, d-2N+1$ and $d-4N+1$ for cup lengths respectively.
\end{remark}
 \section{$S^3$-Equivariant Maps}\label{S^3_index}
 Given a $G$-space we define the homotopy orbit $X_{hG}:=EG\times_{G}X$. Here, $EG$ is the total space for the universal $G$-bundle. Note that if the $G$-action on $X$ is free, $X_{hG}\equiv X/G$. Given a $G$-space $X$ we have a canonical fibration \[X\to X_{hG}\xrightarrow{p} BG.\]
 We the define the Fadell-Husseini index associated to $X$ for a suitable coefficient ring $R$ to be \[\Index_{G}(X;R)=\ker p^*:H^*(BG;R)\to H^*(X_{hG};R).\] One of the interesting properties of the index is the monotonicity property. A $G$-map between two $G$-spaces $X$ and $Y$ will imply  \[\Index_G(Y)\subset \Index_G(X).\]
 We will Denote by $\Index_G^q(X;R)=\Index_G(X;R)\cap H^q(X;R)$. Various cases have been considered in real and complex Stiefel manifolds to rule out $\bb Z_2$ and $S^1$-equivariant maps between them in the articles \cite{Pe13} and \cite{Pe97}. We here considered the quaternionic Stiefel manifolds with $S^3$ action on it.
 \subsection{Index of odd-dimensional sphere}
 Consider the free action of $S^1$ to the odd dimensional sphere $S^{2n-1}$ by complex multiplication which results the fibration \[S^{2n-1}\to \bb CP^n\to \bb CP^{\infty}. \] The only non-trivial differential for the Serre spectral sequence associated with the above fibration is at $E_{2n}$ page. The differential takes the generator of $H^{2n-1}(S^{2n-1};\bb Z_2)$ to the $2n$-dimensional generator of $H^{*}(\bb CP^{\infty};\bb Z_2)$. Thus, 
 \begin{proposition}
 $\Index_{S^1}(S^{2n-1};\bb Z_2)$ is the ideal $\lra{\alpha^{2n}}$ in  $H^*(BS^1;\bb Z_2)$. Here $\alpha$ denotes the two-dimensional generator of the cohomology ring $H^*(BS^1;\bb Z_2)$.
 \end{proposition}
 We obtain the following using a similar argument on free $S^3$ action on the odd-dimensional sphere $S^{4n-1}$. 
 \begin{proposition}\label{index}
 $\Index_{S^3}(S^{4n-1})$ is the ideal $\lra{\alpha^{4n}}$ in  $H^*(BS^3; \bb Z_2)$. Here $\alpha$ denotes the four-dimensional generator of the cohomology ring $H^*(BS^3; \bb Z_2)$.
 \end{proposition}
 This immediately gives
 \begin{corollary}
 There exists $S^3$-map between $S^{4n-1}\to S^{4m-1}$ if and only if $n\leq m$.   
 \end{corollary}
 \begin{theorem}\label{indst}
  a) $\Index_{S^3}(\bb HV_{n,k};\bb Z_2)=\lra{\omega^N}$,
  where $\omega$ is the generators of $\pmod 2$ cohomology algebra of $BS^3$ and $|\omega|=4$.\\
  b) $\Index_{S^3}^{4(n-k+1)}(\bb HV_{n,k};\bb Z)=\bb Z[\binom{m}{m-l+1}k^{n-k+1}]$, where $k$ is the generator of the integral cohomology algebra of $BS^3$ and $|k|=4$.
 \end{theorem}
 \begin{proof}
 From equation \eqref{f3} and using Serre spectral sequence associated with it, Theorem \ref{Quaternion_cohomology} readily gives part $(a)$ of the theorem.\\

 For part $(b)$, we again use the Serre spectral sequence associated with fibration (6) and the pullback square in the proof of the Theorem \ref{Quaternion_cohomology}. Recall that $H^*(\bb HV_{n,k};\bb Z)=\Lambda_{\bb Z}(z_{n-k+1}, z_{n-k+2}, \ldots, z_{n})$ where  $|z_i|=4i-1$. The generator $z_{n-k+1}$ is transgressive, translating the proof of the Theorem \ref{Quaternion_cohomology} in integral set up we get \[d_{4(n-k+1)}(z_{n-k+1})=\binom{n}{n-k+1}k^{n-k+1}.\] This completes the proof.
 \end{proof}
 
We will use the index computations to rule out the existence of $S^3$-equivariant maps between quaternionic Stiefel manifolds and certain $S^3$-spheres. 


 \textit{Proof of Theorem \ref{main3}(a):}
 The existence of such a $S^3$-equivariant map will imply
 \[\Index_{S^3} (\bb HV_{m,l})\subset \Index_{S^3} (\bb HV_{n,k}). \] From part $(b)$ of the Theorem \ref{indst} the proof readily follows.\\

 \textit{Proof of Theorem \ref{main3}(b):}
 There exists $S^3$-map $f:~Sp(n)\to Sp(m)$ if and only if $n~|~m$.  
 The second part of the Theorem (4.4) implies that if there exists $S^3$-map $f:~Sp(n)\to Sp(m)$, then $n$ must divide $m$.\\
 For the `if' part, note that $m=nk$ for some $k\in N^+$  implies that any $M\in Sp(n)$ can be embedded in $Sp(m)$. This can be done by putting $M$ as $n\times n$ block matrix in the diagonal and putting everywhere $0$ otherwise. This completes the proof.\\

 \textit{Proof of Theorem \ref{main3}(c), (d)} is direct consequence of the Theorem \ref{indst}.

{\bf Acknowledgements :} We would like to express our gratitude to Aniruddha Naolekar for introducing us to the problem and for the productive discussion we had with him during our stay at the Indian Statistical Institute, Bangalore Centre, while working on this project. We also thank the anonymous referee for his helpful suggestions that improved the presentation of the manuscript.

\bibliography{main}
\bibliographystyle{IEEEtran}
\end{document}